\documentclass[12pt]{article}
\usepackage{tikz}
\usepackage{float}
\usepackage{hyperref}
\usetikzlibrary{positioning,matrix}
\usepackage[margin=2.5cm]{geometry}
\usepackage[utf8]{inputenc}
\usepackage[T1]{fontenc}
\usepackage{amsmath}
\usepackage{amsfonts}
\usepackage{authblk}
\usepackage{amssymb}
\usepackage{amsthm}
\usepackage{graphicx}
\usepackage{tabularx} 
\usepackage{ragged2e} 
\usepackage{booktabs} 
\newcolumntype{C}{>{\Centering$}X<{$}}
\newtheorem{theorem}{Theorem}[section]     
\newtheorem{corollary}[theorem]{Corollary}
\newtheorem{lemma}[theorem]{Lemma}

\newtheorem{definition}[theorem]{Definition}
\newtheorem{example}[theorem]{Example}

\newtheorem{problem}[theorem]{Problem}

\date{}

\allowdisplaybreaks
\usepackage{fancyhdr}
\fancypagestyle{mypagestyle}{%
	\fancyhf{}
	\fancyhead[OC]{Subarsha Banerjee}
	\fancyhead[EC]{Laplacian Spectra of Comaximal Graph of  $\mathbb{Z}_{n}$}
	\fancyfoot[C]{\thepage}%
}
\title{Laplacian Spectra of Comaximal Graph of $\mathbb{Z}_{n}$}
\author{Subarsha Banerjee}
\affil{Department of Pure Mathematics, University of Calcutta \\ 35 Ballygunge Circular Road, Kol-700019
\\
subarshabnrj@gmail.com
}

\begin{document}
	\maketitle

\begin{abstract}
	
This article focuses on finding the  eigenvalues of the Laplacian matrix of the comaximal graph $\Gamma(\mathbb Z_n)$ of  the  ring $\mathbb Z_n$ for $n> 2$.
We  determine the eigenvalues of $\Gamma(\mathbb Z_n)$ for various $n$ and also provide a  procedure to find the eigenvalues of $\Gamma(\mathbb Z_n)$ for any $n> 2$.
We show that  $\Gamma(\mathbb Z_n)$ is Laplacian Integral for $n=p^\alpha q^\beta$ where $p,q$ are primes and $\alpha, \beta$ are non-negative integers. 
The algebraic and vertex connectivity of $\Gamma(\mathbb Z_n)$ have been shown to be equal for all $n> 2$.
An upper bound on the second largest eigenvalue of $\Gamma(\mathbb Z_n)$ has been obtained and  a necessary and sufficient condition for its equality has also been determined.
Finally we discuss the multiplicity of the spectral radius and the multiplicity of the algebraic connectivity of $\Gamma(\mathbb Z_n)$.
Some problems have been discussed at the end of this article for further research.

\end{abstract}
\textbf{Keywords:}  comaximal graph; laplacian eigenvalues; vertex connectivity; algebraic connectivity; laplacian spectral radius; finite ring
\\
\textbf{2010 Mathematics Subject Classification:} 05C25, 05C50.

\section{Introduction}
Let $G$ be a finite simple undirected graph of order $n$ with vertex set $V(G)=\{v_1,v_2,\ldots,v_n\}$.
Note that $v_i\sim v_j$ denotes that  $v_i$ adjacent to $v_j$ for $1\le i\neq j\le n$.
The \textit{adjacency matrix} of $G$  denoted by  $A(G)=(a_{ij})$ is an $n\times n$ matrix defined as $a_{ij}=1$ when $v_i\sim v_j$ and $0$ otherwise.
The \textit{Laplacian matrix} $L(G)$ of $G$ is defined as $L(G)=D(G)-A(G)$ where $A(G)$ is the adjacency matrix of $G$  and $D(G)$ is the diagonal matrix of vertex degrees.
Since the  matrix $L(G)$ is a real, symmetric and a positive semi-definite matrix, all its eigenvalues are real and  non-negative. 
Also  $0$ is an eigenvalue of $L(G)$ with eigenvector $[1,1,1,\ldots,1]^T$ whose multiplicity  equals the number of connected components in the graph $G$.
Let the  eigenvalues of $L(G)$ be denoted by  $\lambda_1\geq \lambda_2 \geq \cdots \geq \lambda_{n-1}\geq\lambda_n=0$.
The largest eigenvalue $\lambda_1$ is known as the \textit{spectral radius} of $G$ and the second smallest eigenvalue $\lambda_{n-1}$ is known as the \textit{algebraic connectivity} of $G$.
Also $\lambda_{n-1}>0$ if and only if $G$ is connected.
The term \textit{algebraic connectivity} was given by Fiedler in \cite{fiedler1973algebraic}.
A \textit{separating set} in a connected graph $G$ is a set $S\subset V(G)$ such that $V(G)\setminus S$ has more than $1$ connected component.
The \textit{vertex connectivity} of  $G$ denoted by $\kappa(G)$  is defined as $\kappa(G)=\min \{|S|: S \text{ is a separating set of  } G\}$.
The papers \cite{fiedler1973algebraic} and \cite{de2007old} list several interesting  properties of $\lambda_{n-1}$ and $\kappa$. 
Readers may refer to  \cite{merris1994laplacian} for a survey on  $L(G)$ of a graph $G$.
A graph $G$ whose adjacency matrix has all its eigenvalues as integers is known as \textit{integral} graph.
Harary and Schwenk  first posed the question "Which graphs have integral spectra" in \cite{harary1974graphs}.
Given a positive integer $n$, it is difficult to locate the graphs with integral spectra among the graphs having $n$ number of vertices.
For a survey and more detailed information about integral graphs, the readers are referred to \cite{balinska2002survey}.
A graph $G$ is called \textit{Laplacian integral} if all the eigenvalues of $L(G)$ are integers.
\\
\\
Let $R$ be a commutative ring with unity.
The \textit{comaximal graph} of a ring $R$ denoted by $\Gamma (R)$ was introduced by Sharma and Bhatwadekar in \cite{sharma1995note}.
The vertices of the graph $\Gamma(R)$ are  the elements of the ring $R$ and any two distinct vertices $x,y$ of $\Gamma(R)$ are adjacent if and only if $Rx+Ry=R$.
They  proved that $R$ is a finite ring if and only if the \textit{chromatic number} of $\Gamma(R)$ denoted by $\chi(\Gamma(R))$ is finite.
It was further shown  that  $\chi(\Gamma(R))$  satisfies $\chi(\Gamma(R))=t+l$ where $t$ denotes the number of \textit{maximal ideals} of $R$ and $l$  denotes the number of \textit{units} of  $R$.
A lot of research has been done on the comaximal graph of a ring $R$ over the last few decades.
For some literature on $\Gamma(R)$, readers may refer to the works \cite{maimani2008comaximal,moconja2011structure} and \cite{samei2014comaximal}.  
\\
\\
Let $\mathbb{Z}_n$ denote the ring of integers modulo $n$ where $n>2$.
Let the  elements of  $\mathbb{Z}_n$ be denoted by  $0,1,2,\ldots,n-2,n-1$.
The number of integers prime to $n$ and less than $n$ is denoted by \textit{Euler Totient} function $\phi(n)$.
We say that $d$ is a \textit{proper} divisor of $n$ if $d$ divides $n$ and $d$ does not equal $1$ or $n$.
A \textit{subring} $I$ of a ring $R$ is said to be an \textit{ideal} of $R$ if for all $r\in R$ and $i\in I$, $ri,ir\in I$.
The \textit{ideal generated by $r$} is defined to be the smallest ideal of $R$ containing $r$ and is denoted by $\langle r \rangle$.
A ring $R$ is said to be a \textit{Principal Ideal Ring (PIR)} if every ideal $I$ of $R$ is of the form $I=\langle r\rangle$ for some $r\in R$.
A ring $R$ is said to be an \textit{integral domain} if $xy=0$ in $R$ implies either $x=0$ or $y=0$. 
A ring $R$ is said to be a \textit{Principal Ideal Domain (PID)}  if it is a PIR and an integral domain.
For any given set $A$,  $|A|$ denotes the number of elements in the set $A$.
\\
\\
By a \textit{null} graph on $n$ vertices denoted by $\Theta_n$, we shall mean a graph having $n$ vertices and zero edges.
An \textit{empty} graph means a graph with no vertices and edges.
The \textit{join} of two graphs $G_1=(V_1,E_1)$ and $G_2=(V_2,E_2)$ denoted by  $G_1\lor G_2$ is a graph obtained from $G_1$ and $G_2$ by joining each vertex of $G_1$ to all vertices of $G_2$.
The \textit{union} of two graphs $G_1=(V_1,E_1)$ and $G_2=(V_2,E_2)$ denoted by $G_1\cup G_2$  is the graph with vertex set $V=V_1 \cup V_2$ and edge set $E=E_1 \cup E_2$.
For various standard terms related to   ring theory and graph theory  used in the article, the readers are referred to the texts   \cite{dummit2004abstract} and  \cite{diestel2000graduate}  respectively.
\\
\\
Throughout this article by \textit{eigenvalues} and \textit{characteristic polynomial} of $\Gamma(\mathbb{Z}_n)$, we shall mean the eigenvalues and characteristic polynomial of $L(\Gamma(\mathbb{Z}_n))$.
The set of all eigenvalues of  $G$ is denoted by $\sigma(G)$.
We denote the characteristic polynomial of  $G$ by $\mu(G,x)$.
\subsection{Arrangement of the Article}
In \textbf{Section} \ref{Sec2} we provide the preliminary  theorems that have been used throughout the article.
In  \textbf{Section} \ref{Sec3} we discuss the structure of $\Gamma(\mathbb{Z}_n)$ and express it in terms of its subgraphs [Equation \ref{join}].
Using it we find the \textit{characteristic polynomial} of $\Gamma(\mathbb{Z}_n)$ for $n> 2$ [Theorem \ref{Th1}].
In \textbf{Subsection} \ref{Any} of \textbf{Section} \ref{Sec3} we provide the procedure to find the spectra of $\Gamma(\mathbb{Z}_n)$ for any $n> 2$.
We then  explicitly determine the \textit{spectra}  of $\Gamma(\mathbb{Z}_n)$ in \textbf{Subsection} \ref{Integral} for  $n=p^\alpha q^\beta$ where $p,q$ are distinct primes and $\alpha,\beta$ are non-negative integers [Theorems \ref{pm}, \ref{prod}] and conclude that $\Gamma(\mathbb{Z}_n)$ is Laplacian Integral for $n=p^\alpha q^\beta$.
In \textbf{Section} \ref{Sec4} we discuss  \textit{vertex connectivity} and \textit{algebraic connectivity } of  $\Gamma(\mathbb{Z}_n)$ and show that  $\kappa(\Gamma(\mathbb{Z}_n))= 
\lambda_{n-1}(\Gamma(\mathbb{Z}_n))$ for all $n> 2$ [Theorem \ref{VC=AC}].
In \textbf{Section} \ref{Sec5}, we find an upper bound of  the second largest eigenvalue of $\Gamma(\mathbb Z_n)$  and determine a necessary and sufficient condition when it attains its bounds
[Theorem \ref{Th7}].
We use it to determine the multiplicity of the spectral radius of $\Gamma(\mathbb Z_n)$ [Theorem \ref{mult spectral}].
We also determine the multiplicity of the algebraic connectivity of $\Gamma(\mathbb Z_n)$ [Theorem \ref{multAC}]  and provide certain results regarding the connectivity of an induced subgraph of $\Gamma(\mathbb Z_n)$ in subsection \ref{Sec50}.
Finally in \textbf{Section} \ref{Sec6}  we provide some problems for further research.

\section{Preliminaries}
\label{Sec2}
In this section we will provide some preliminary  theorems that will
be required in our subsequent sections. 
\begin{theorem}[Corollary 3.7 of \cite{mohar1991laplacian}]
	\label{Thjoin}
	Let $G_1\lor G_2$ denote the join of two graphs $G_1$ and $G_2$.
	Then $$\mu(G_1\lor G_2,x)=\frac{x(x-n_1-n_2)}{(x-n_1)(x-n_2)}\mu(G_1,x-n_2)\mu(G_2,x-n_1)$$
	where $n_1$ and $n_2$ are orders of $G_1$ and $G_2$  respectively.
\end{theorem}

\begin{theorem}[Theorem $3.1$ of \cite{mohar1991laplacian}]\label{Thunion}
	Let $G$ be the disjoint union of the graphs $G_1,G_2,\ldots, G_k$. Then
	$$\mu(G,x)=\prod_{i=1}^{k}\mu(G_i,x).$$
\end{theorem}

\begin{theorem}[Theorem 2.2 of \cite{mohar1991laplacian}]
	\label{Th3}
If $G$ is a simple graph on $n$ vertices then the largest eigenvalue $\lambda_1$ of $G$ satisfies $\lambda_1\le n$, where the equality holds  if and only if its complement $G^c$ is disconnected.
	
\end{theorem}

\begin{definition}[Definition 3.9.1 of \cite{cvetkovic2009introduction}]
\label{Def1}
Given a graph $G$ with vertex set $V(G)$, a partition $V(G)=V_1 \cup V_2 \cup \dots \cup V_k$ is said to be an \textit{equitable partition} of $G$ if every vertex in $V_i$ has the same number of neighbors $b_{ij}$ in $V_j$ where $1\le i,j\le k$.
\end{definition}

\begin{theorem} [Theorem $2.1$ of \cite{kirkland2002graphs}]
	\label{VCequalsAC}
	Let $G$ be a non-complete, connected graph on $n$ vertices.
	Then $\kappa(G)=\lambda_{n-1}(G)$  if and only if $G$ can be written as $G=\mathcal{G}_1\lor \mathcal{G}_2$, where $\mathcal{G}_1$ is a disconnected graph on $n-\kappa(G)$ vertices and $\mathcal{G}_2$ is a graph on $\kappa(G)$ vertices with $\lambda_{n-1}(\mathcal{G}_2)\geq 2\kappa(G)-n$. 
\end{theorem}

\begin{definition}[See Page $15$ of \cite{schwenk1974computing}]
	\label{generalizedcomposition}
	Let $H$ be a graph with vertex set $V(H)=\{1,2,\ldots, k\}$.
	Let $G_i$ be disjoint graphs of order $n_i$  with vertex sets $V(G_i)$ where $1\le i\le k$.
	The H-join of graphs $G_1,G_2,\dots, G_k$ denoted by $H[G_1,G_2,\dots,G_k]$ is formed by taking the graphs $G_i$  and any two vertices $v_i\in G_i$ and $v_j\in G_j$ are adjacent if $i$ is adjacent to $j$ in $H$. 
\end{definition}

\begin{theorem}[Theorem $8$ of \cite{cardoso2013spectra}]
	\label{spectrajoin}
	Let us consider  a family of $k$ graphs $G_j$ of order $n_j$, with $j\in \{1,2,\ldots,k\}$ having Laplacian spectrum $\sigma(G_j)$.
	If $H$ is a graph such that $V(H)=\{1,2\ldots,k\}$, then the Laplacian spectrum of  $H[G_1,G_2\ldots,G_k]$ is given by 
	$$\sigma\bigg(H[G_1,G_2\ldots,G_k]\bigg)=\bigg(\bigcup_{j=1}^k(N_j+\sigma(G_j)\setminus\{0\})\bigg)\cup \sigma(M)$$
	where 
	
	\begin{equation}
	\label{matrix}
	M=
	\left[
	\begin{array}{ccccccccccccccccc}
	N_1 & -\rho_{1,2}& \ldots & -\rho_{1,k}
	\\
	-\rho_{1,k} & N_2 & \ldots & -\rho_{2,k}
	\\
	\vdots & \vdots &\ddots 
	\\
	-\rho_{1,k}& -\rho_{2,k} &\ldots & N_k
	\end{array}
	\right],
	\end{equation}
	
	$\rho_{a,b}=\rho_{b,a}=\begin{cases}
	\sqrt{n_an_b}& \text{ if } lq\in E(H)
	\\
	0 & \text{ otherwise} 
	\end{cases}$
	and
	$N_j=\begin{cases}
	\sum_{i\in N_H(j)} n_i & \text{ if } N_H(j)\neq \emptyset\\
	0 & \text{ otherwise}.
	\end{cases}$
	\\
	Here $N_H(j)=\{i: ij\in E(H)\}$.
\end{theorem}

\section{Structure of  $\Gamma(\mathbb{Z}_n)$ and its Laplacian Spectra }
\label{Sec3}
In this section we first give a description of the structure of $\Gamma(\mathbb{Z}_n)$.
We show that $\Gamma(\mathbb{Z}_n)$ can be expressed as the join and union of certain sub-graphs of $\Gamma(\mathbb{Z}_n)$.
We then investigate Laplacian Spectra of $\Gamma(\mathbb{Z}_n)$ for various $n$.
\\
We first find an equivalent condition for adjacency of two vertices in  $\Gamma(\mathbb{Z}_n)$.
\\
\\
Using the  adjacency criterion for any two vertices in $\Gamma(R)$, we find that two vertices $v_i,v_j\in \Gamma(\mathbb{Z}_n)$ are adjacent if and only if $\mathbb{Z}_n v_i+\mathbb Z_n v_j=\mathbb{Z}_n$.
Now since $\mathbb{Z}_n$ is a \textit{Principal Ideal Domain}, so $\mathbb{Z}_n v_i=\langle v_i\rangle$.
Thus the adjacency criterion in $\Gamma(\mathbb{Z}_n)$ becomes the following:
\begin{align}
\label{criterion}
\begin{split}
v_i \text{ is adjacent to } v_j  \text{ in } \Gamma(\mathbb{Z}_n) \iff \langle v_i\rangle +\langle v_j\rangle =\mathbb{Z}_n \\
\text{ where } \langle v_i\rangle ,\langle v_j\rangle \text{ denote the ideal generated by } v_i, v_j \text{ respectively }.
\end{split}
\end{align}

If $V(\Gamma(\mathbb{Z}_n))$ denotes the set of vertices of $\Gamma(\mathbb{Z}_n)$ then $V(\Gamma(\mathbb{Z}_n))=\mathcal{S}\cup \mathcal{T}$ where $\mathcal{S}= \{a:\gcd(a,n)=1\}$ and $\mathcal{T}=\mathbb Z_n\setminus \mathcal{S}$. 
Thus the vertex set of $\Gamma(\mathbb{Z}_n)$ is $\mathcal{S}\cup \mathcal{T}$.
The following observations about $\Gamma(\mathbb{Z}_n)$ can be easily made.
\begin{lemma}\label{Lemma1}
	
If $v\in \mathcal{S}$ then $v$ is adjacent to $w$ for all   $w\in \Gamma(\mathbb{Z}_n)$ and hence  $\deg(v)=n-1$.
	
\end{lemma}
\begin{proof}
	Since $v$ is an unit in $\mathbb{Z}_n$, so $\langle v\rangle =\mathbb{Z}_n$.
	Thus for all   $w(\neq v)\in \Gamma(\mathbb{Z}_n)$, $\langle v\rangle+\langle w\rangle =\mathbb{Z}_n+\langle w\rangle =\mathbb Z_n$ and hence $\deg(v)=n-1$.
\end{proof}

\begin{lemma}\label{Lemma2}
	
The vertex $0\in \mathcal T$ is adjacent only to the members of $\mathcal{S}$ and hence $\deg(0)=|\mathcal{S}|=\phi(n)$.

\end{lemma}

Let $G_1$ denote the \textit{induced subgraph} of $\Gamma(\mathbb{Z}_n)$ on the set $\mathcal{S}$ and $G_2'$ denote the  \textit{induced subgraph} of $\Gamma(\mathbb{Z}_n)$ on the set $\mathcal{T}$.
Since $\mathcal{S}$ has $\phi(n)$ elements, using Lemma \ref{Lemma1} we find that $G_1$ is the complete graph on $\phi(n)$ vertices and hence $G_1\cong K_{\phi(n)}$.
Also using Lemma \ref{Lemma1} we find that  every vertex of $\Gamma(\mathbb{Z}_n)$ in $\mathcal{S}$ is adjacent to all the vertices of $\mathcal{T}$, which makes us conclude that 
\begin{equation*}
\Gamma(\mathbb{Z}_n)=G_1\lor G_2'\cong K_{\phi(n)} \lor G_2'
\end{equation*}
where $G_2'$ is a graph on $n-\phi(n)$ vertices.
	
Again using Lemma \ref{Lemma2} we find that $G_2'$ is the union of graphs $\Theta_1$  and $G_2$ where $G_2$ is a graph on $n-\phi(n)-1$ vertices and $\Theta_1$ denotes the null graph on the set $\{0\}$.
The graph $G_2$ is the induced sub-graph of  
$\Gamma(\mathbb{Z}_n)$ on the set $\mathcal{T}\setminus\{0\}$.
Thus we have
\begin{equation}\label{join}
\Gamma(\mathbb{Z}_n)\cong K_{\phi(n)} \lor G_2'\cong K_{\phi(n)}\lor(G_2\cup \Theta_1).
\end{equation}
We shall use Equation  \ref{join} repeatedly in our results.

\begin{theorem}\label{Th1}

The characteristic polynomial of $\Gamma(\mathbb{Z}_n)$ is $\mu(\Gamma(\mathbb{Z}_n))=x(x-n)^{\phi(n)}\mu(G_2,x-\phi(n))$
where $G_2$ is given by Equation \ref{join} .
		
\end{theorem}
	
\begin{proof}
		
Using Equation \ref{join} and Theorems  \ref{Thjoin} and  \ref{Thunion} we obtain,
\begin{equation}\label{main}
\begin{split}
\mu(\Gamma(\mathbb{Z}_n))=\mu(G_1\lor G_2')=\mu( K_{\phi(n)} \lor G_2')\\
=\frac{x(x-n)}{(x-\phi(n))\bigg(x-(n-\phi(n))\bigg)}\mu(K_{\phi(n)},x-(n-\phi(n))) \mu(G_2',x-\phi(n))\\
=\frac{x(x-n)}{(x-\phi(n))\bigg(x-(n-\phi(n))\bigg)} \bigg (x-(n-\phi(n))\bigg)\bigg (x-n\bigg )^{\phi(n)-1} \mu(G_2',x-\phi(n))\\
=\frac{x(x-n)^{\phi(n)}}{(x-\phi(n))} \mu(G_2',x-\phi(n))
=x(x-n)^{\phi(n)}\mu(G_2,x-\phi(n)).
\end{split}
\end{equation}

\end{proof}
The following observation about $\mu(\Gamma(\mathbb{Z}_n))$ can be easily made

\begin{corollary}
If $n> 2$, then $n$ is an eigenvalue of $\Gamma(\mathbb Z_n)$ with multiplicity atleast $\phi(n)$.
\end{corollary}

\begin{corollary}\label{prime}
	If $n=p$ where $p$ is a prime number,  then $p$ and $0$ are eigenvalues of $\Gamma(\mathbb Z_n)$ with multiplicity $p-1$ and $1$ respectively.
\end{corollary}

\begin{proof}
When $n=p$ is a prime number, then $\phi(n)=n-1$ and hence the set $\mathcal{T}\setminus\{0\}$ is empty which implies that the graph $G_2$ is the \textit{empty} graph.
Thus using Equation \ref{main}, 
\begin{equation*}
	\mu(\Gamma(\mathbb Z_n))=x(x-n)^{n-1}=x(x-p)^{p-1}.
\end{equation*}
\end{proof}

From Equation \ref{main} of Theorem \ref{Th1} we find that the eigenvalues of $\Gamma(\mathbb{Z}_n)$ are known if the spectra of the graph $G_2$ given in Equation \ref{join}  is completely determined.
We thus proceed to study the graph $G_2$ in more detail.
\subsection{Structure of $G_2$}
Let $n=p_1^{\alpha_1}p_2^{\alpha_2}\cdots p_k^{\alpha_k}$ be a prime factorization of $n$ where $p_1<p_2<\dots<p_k$ are primes and $\alpha_i$ are positive integers.
\\
The total number of positive divisors of $n$ is given by $\sigma(n)=(\alpha_1+1)(\alpha_2+1)\cdots (\alpha_k+1)$.
The total number of \textit{proper} positive divisors of $n$  will be given by $\sigma(n)-2$.
We denote $\sigma(n)-2$ by $w$.
\\
\\
Let $d_1<d_2<\cdots<d_w$ be the set of all  proper divisors of $n$ arranged in increasing order.
For each $d_i$ where $1\le i\le w$ we define 
\begin{align}\label{divisor}
A_{d_i}=\{x:\gcd(x,n)=d_i\}.
\end{align}

Any element of $A_{d_i}$ is of the form $zd_i$ where $\gcd(z,\frac{n}{d_i})=1$ and hence number of elements of $A_{d_i}$ is $\phi(\frac{n}{d_i})$.
Thus $|A_{d_i}|=\phi(\frac{n}{d_i})$.
Clearly $V(G_2)=\cup_{i=1}^w A_{d_i}$.
\begin{lemma}
	\label{Adjacency}
	
$x_i\in A_{d_i}$ is adjacent to $x_j\in A_{d_j}$ if and only if $\gcd(d_i,d_j)=1$.	
	
\end{lemma}

\begin{proof}
	
Assume that $x_i\in A_{d_i}$ is adjacent to $x_j\in A_{d_j}$.
Using Equation \ref{criterion}, $x_i$ adjacent to $x_j$ implies 
$\langle x_i \rangle +\langle x_j \rangle =\mathbb{Z}_n$ which in turn implies that either $\gcd(x_i,x_j)=1$ or $\gcd(x_i,x_j)$ is an unit in $\mathbb{Z}_n$.
We consider the following two cases:
\begin{itemize}
	\item [Case 1:] $\gcd(x_i,x_j)=1$
	\\
	Let $\gcd(d_i,d_j)=d$, then $d\vert d_i, d\vert d_j$.
	Since $\gcd(x_i,n)=d_i$ and $\gcd(x_j,n)=d_j$, we have $d_i\vert x_i$ and $d_j\vert x_j$ which in turn implies that $d\vert x_i$ and $d\vert x_j$.
	Again since $\gcd(x_i,x_j)=1$,  $d=1$ and hence  $\gcd(d_i,d_j)=1$.	
	
	\item[Case 2:] $\gcd(x_i,x_j)$ is an unit in $\mathbb{Z}_n$.
	\\
	Let $\gcd(x_i,x_j)=a$ which is an unit in $\mathbb{Z}_n$ and hence $\gcd(a,n)=1$.
	Let $\gcd(d_i,d_j)=d$, then $d\vert d_i, d\vert d_j$.
	Since $\gcd(x_i,n)=d_i$ and $\gcd(x_j,n)=d_j$, we have $d_i\vert x_i, d_i
	\vert n$ and $d_j\vert x_j, d_j\vert n$ which in turn implies that $d\vert x_i$, $d\vert x_j$ and $d\vert n$.
	Since $\gcd(x_i,x_j)=a$, $d\vert a$.
	Since $\gcd(a,n)=1$, from the facts that $d\vert a$ and $d\vert n$ it follows that  $d=1$. Hence  $\gcd(d_i,d_j)=1$.	
	
\end{itemize}
Thus if $x_i\in A_{d_i}$ is adjacent to $x_j\in A_{d_j}$, then $\gcd(d_i,d_j)=1$.
\\
\\
Conversely, we now assume that $\gcd(d_i,d_j)=1$. 
Let $d=\gcd(x_i,x_j)$. We claim that either $d=1$ or $d$ is an unit in $\mathbb{Z}_n$.
Assume the contrary, then $d>1$ and $d$ is not an unit in $\mathbb{Z}_n$ which implies $d(>1)$ divides $n$. 
\begin{align*}
\begin{split}
\text{ If } d=\gcd(x_i,x_j) \text{ then } d\vert x_i, d\vert x_j
\\
\implies (d \vert x_i, d\vert n) \text{ and } (d\vert x_j, d\vert n)
\implies d \vert d_i=\gcd(x_i,n) \text{ and } d \vert d_j=\gcd(x_j,n)  
\\
\implies d\vert \gcd(d_i,d_j)=1 \text{ which is a contradiction }
\end{split}
\end{align*}
Thus either $d=1$ or $d$ is an unit in $\mathbb{Z}_n$ and hence $\langle x_i\rangle +\langle x_j\rangle =\langle d\rangle=\mathbb{Z}_n$ which implies by Equation \ref{criterion} that $x_i\in A_{d_i}$ is adjacent to $x_j$ in $A_{d_j}$.
\end{proof}

\begin{lemma}\label{Lem1}
	If $v_i\in A_{d_i}$ is adjacent to $v_j\in A_{d_j}$ for some $i\neq j$, then $v_i$ is adjacent to $v_j$ for all $v_j\in A_{d_j}$.  
\end{lemma}

\begin{proof}
	Let $v_i\in A_{d_i}$ is adjacent to $v_j\in A_{d_j}$ for some $i\neq j$, then using Lemma \ref{Adjacency} $\gcd(d_i,d_j)=1$.
	Let $v_j'\neq v_j$ be another member of $V_j$, then $\gcd(v_j',n)=d_j$. Using the fact that $\gcd(d_i,d_j)=1$ and Lemma \ref{Adjacency}, we conclude that $v_j'$ is adjacent to $v_i$.
\end{proof}

\begin{lemma}\label{Lem2}
	No two members of the set $A_{d_i}$ are adjacent. 
\end{lemma}

\begin{proof}
If $v_i, v_j\in A_{d_i}$ then $\gcd(v_i,n)=\gcd(v_j,n)=d_i$.
Using Lemma \ref{Adjacency}, the proof follows.

\end{proof}

If $v_i\in A_{d_i}$, using Lemma \ref{Lem1}  we observe that the number of neighbors of $v_i$ in $A_{d_j}$ where $j\neq i$ is fixed, i.e.  either the number of neighbors of $v_i$ in $A_{d_j}$ equals $0$ or $|A_{d_j}|$.
Also using Lemma \ref{Lem2}, the number of neighbors of $v_i$ in $A_{d_i}$ equals $0$ for all $1\le i\le w$.
If we denote $V_i=A_{d_i}$ where $1\le i\le w$, then using Definition \ref{Def1} we find that $V_1\cup V_2\cup \dots \cup V_w$  is an \textit{equitable partition}  of the graph $G_2$.
\\
Thus we have the following theorem;
\begin{theorem}\label{Equitable Partition}
For any $n\ge 2$, the induced subgraph $G_2$ of $\Gamma(\mathbb Z_n)$ with vertex set $V(G_2)$ has an  equitable partition  as $V(G_2)=\cup_{i=1}^w A_{d_i}$, where $w$ denotes the total number of positive proper divisors of $n$ and the sets $A_{d_i}$ have been defined as in Equation \ref{divisor}.

\end{theorem}

\subsection{Laplacian Spectra of $\Gamma(\mathbb{Z}_n)$ for any $n> 2$}
\label{Any}
In this section we provide the procedure to find the eigenvalues of $\Gamma(\mathbb{Z}_n)$ for any $n> 2$.
Using Equation \ref{main} of Theorem \ref{Th1}, we find that determining the spectra of $\Gamma(\mathbb{Z}_n)$ for any $n>2$ boils down to finding the spectra of induced subgraph $G_2$ of $\Gamma(\mathbb{Z}_n)$.
\\
\\
Using Theorems \ref{Equitable Partition} and \ref{generalizedcomposition}, it is evident that $G_2$ is the $H$-join of the graphs $G_{d_i}$ where $G_{d_i}$ is the induced subgraph of $\Gamma(\mathbb Z_n)$ on $A_{d_i}$, and $H$ can be obtained as follows:
\paragraph{Construction of $H:$}
\label{para}
$V(H)=\{d_i:1\le i\le w \text{ where } d_i \text{ is a positive proper divisor of } n\}$.
The vertices $d_i,d_j$ are adjacent in $H$ if and only if $\gcd(d_i,d_j)=1$.
\\
Thus $E(H)=\{d_id_j: \gcd(d_i,d_j)=1\}.$
\\
\\
We use Theorem \ref{spectrajoin} to determine the spectra of $G_2$.
We find that $G_2$ is the $H$- join of $G_{d_i}$, where $G_{d_i}$ is a null graph on $\phi(\frac{n}{d_i})$ vertices.
Hence $\sigma(G_{d_i})=\{0\}.$
\\
Also, $$N_{H}(d_j)=\{d_i: \gcd(d_i,d_j)=1\}$$ and hence  
$$N_{d_j}=\sum_{d_i\in N_H(d_j)} n_i=\sum_{d_i:\gcd(d_i,d_j)=1} \phi(\frac{n}{d_i}).$$
Moreover, $n_{d_i}=\phi(\frac{n}{d_i})$, where $1\le i\le w$.

\begin{example}
	If $n=pqr$ where where $p,q,r$ are primes with $p<q<r$, then the proper positive divisors of $n$ are $p,q,r,pq,pr,qr$. 
	Using the construction of $H$  given above, we find that $G_2$ is the $H$ join of $G_p,G_q,G_r,G_{pq},G_{pr},G_{qr}$ where $H$ is given by(Figure \ref{H})
	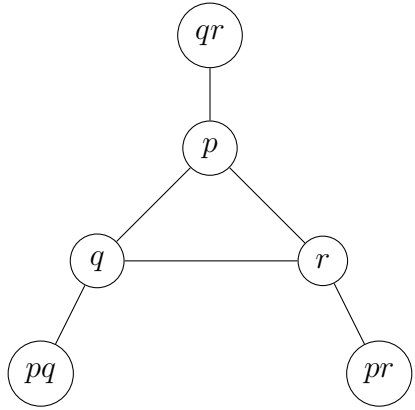
\begin{figure}[H]
		\centering
		\begin{tikzpicture}[scale=0.75]
		
		\node[shape=circle,draw=black] (1) at (0,0) {$p$};  
		\node[shape=circle,draw=black] (2) at (-2,-2)  {$q$}; 
		\node[shape=circle,draw=black] (3) at  (2,-2) {$r$};  
		\node[shape=circle,draw=black] (5) at  (-3,-4) {$pq$};  
		\node[shape=circle,draw=black] (6) at  (3,-4) {$pr$};  
		\node[shape=circle,draw=black] (4) at (0,2) {$qr$};

		\draw (1) -- (2); \draw (3) -- (2);\draw (1) -- (3);\draw (1) -- (4);\draw (5) -- (2);\draw (3) -- (6);
		\end{tikzpicture}
		\caption{$H$ for $n=pqr$}
		\label{H}
	\end{figure}
	
	Now we have,
	
	\begin{align}
	\label{N}
	\begin{split}
	N_p=\phi(\frac{pqr}{q})+\phi(\frac{pqr}{r})+\phi(\frac{pqr}{qr})
	\\
	=\phi(pr)+\phi(pq)+\phi(p)
	=(p-1)(r-1)+(p-1)(q-1)+(p-1)
	\\
	=(p-1)\{r-1+q-1+1\}=(p-1)(q+r-1).
	\\
	\text{ Similarly, }
	N_q=(q-1)(p+r-1),
	N_r=(r-1)(p+q-1)
	\\
	N_{pq}=(p-1)(q-1),
	N_{pr}=(p-1)(r-1) \text{ and } 	 N_{qr}=(q-1)(r-1)
	\end{split}
	\end{align}
	Also 
	\begin{equation}
	\label{n}
	\begin{split}
	n_p=(q-1)(r-1), n_q=(p-1)(r-1), n_r=(p-1)(q-1)
	\\
	n_{pq}=r-1, n_{pr}=q-1, n_{qr}=p-1.
	\end{split}
	\end{equation}
	
	Using Theorem \ref{spectrajoin} we find that the eigenvalues of $G_2$ are $(p-1)(q+r-1)$ with multiplicity $qr-r-q$, $(q-1)(p+r-1)$ with multiplicity $pr-r-p$, $(r-1)(p+q-1)$ with multiplicity $pq-p-q$, $(p-1)(q-1)$ with multiplicity $r-2$, $(p-1)(r-1)$ with multiplicity $q-2$ and $(q-1)(r-1)$ with multiplicity $p-2$ and remaining eigenvalues  are the eigenvalues of $6\times 6$ matrix $M$ (Equation \ref{matrix})whose entries can be determined from Equations \ref{N} and \ref{n}.
	
\end{example}

	Though the procedure given in Subsection \ref{Any} allow us to determine the spectra of $\Gamma(\mathbb{Z}_n)$ for any $n>2$, it is difficult to use it when number of divisors of $n$  is large as otherwise we would have to deal with large matrices which would make the calculations difficult.
\\
\\
In the next subsection we calculate the spectra of  $\Gamma(\mathbb Z_n)$ for some specific $n$ which in turn helps us observe some interesting facts about $\Gamma(\mathbb Z_n)$.
\subsection{$\Gamma(\mathbb Z_n)$ is Laplacian Integral for $n=p^\alpha q^\beta$}
\label{Integral}
In the next two theorems namely Theorem \ref{pm} and \ref{prod},  we find the spectra of  $\Gamma(\mathbb Z_n)$ for $n=p^\alpha q^\beta$ which in turn helps us conclude that $\Gamma(\mathbb Z_n)$ is \textit{Laplacian Integral} for $p,q$ are primes where $p<q$ and $\alpha, \beta$ are non-negative integers.

\begin{theorem}\label{pm}
	When $n=p^m$ where $p$ is a prime and $m>1$ is a positive integer, then the eigenvalues of $\Gamma(\mathbb Z_n)$ are $n$ with multiplicity $\phi(n)$, $\phi(n)$ with multiplicity $n-\phi(n)-1$ and $0$ with multiplicity $1$.
\end{theorem}

\begin{proof}
	When $p$ is a prime and $m>1$ is a positive integer, the proper divisors of $p^m$ are \\
	$p, p^2, p^3, \ldots, p^{m-2}, p^{m-1}.$
	We partition the vertex set $V(G_2)$ of $G_2$ as $V_1,V_2,\ldots, V_{m-2}, V_{m-1}$\\ where $V_{i}=A_{p^i}=\{x:\gcd(x,n)=p^i\}$.
	\\
	\\
	Since $\gcd(p^i,p^j)=p^{\min\{i,j\}} \neq 1$,  using Lemmas \ref{Adjacency} and \ref{Lem2} we find that $x_i\in V_i$ is not adjacent to $x_j\in V_j$ for all $1\le i,j \le m-1$.
	\\
	Thus no two vertices in the graph $G_2$ are adjacent and hence $G_2=\Theta_{n-\phi(n)-1}$.
	Using Equation \ref{main} we obtain
	$ \mu(\Gamma(\mathbb Z_n))=x(x-n)^{\phi(n)}\bigg(x-\phi(n)\bigg)^{n-\phi(n)-1}.$

\end{proof}
\begin{theorem}\label{prod}
	
	If $n=p^\alpha q^\beta$ where $p,q$ are primes with $p<q$ and $\alpha, \beta$ are positive integers, then the eigenvalues of $\Gamma(\mathbb Z_n)$ are $n$ with multiplicity $\phi(n)$, $(t+1)(p-1)+\phi(n)$ with multiplicity $(t+1)(q-1)-1$, $(t+1)(q-1)+\phi(n)$ with multiplicity $(t+1)(p-1)-1$, $\phi(n)$ with multiplicity $t+1$ and $(t+1)(p+q-2)+\phi(n),0$ each with multiplicity $1$ where $t=p^{\alpha-1}q^{\beta-1}-1$. 
	
\end{theorem}

\begin{proof}
	If $n=p^\alpha q^\beta$ where $p,q$ are primes with $p<q$ and $\alpha, \beta$ are positive integers, then the proper divisors of $n$ are $p^iq^j$ where $0\le i\le \alpha$, $0\le j\le \beta$ with $i+j\notin \{0,\alpha+\beta\}$.
	We partition the vertex set $V(G_2)$ as 
	\begin{align}
	\label{partition}
	\begin{split}
	V(G_2)=\bigg( A_{p}\cup A_{p^2}\cup \cdots \cup A_{p^\alpha}\bigg)\cup \bigg ( A_{q}\cup A_{q^2}\cup \cdots \cup A_{q^\beta}\bigg )
	\\
	\cup \bigg( \cup_{j=1}^\beta A_{pq^j} \bigg )\cup \bigg( \cup_{j=1}^\beta A_{p^2q^j}\bigg )\cup \cdots \cup\bigg (\cup_{j=1}^{\beta-1} A_{p^\alpha q^j}\bigg ).
	\end{split}
	\end{align}

	If $1\le i\le \alpha$, $1\le j\le \beta$, then $\gcd(p^i,q^j)=1$.
	Using Lemma \ref{Adjacency} and  \ref{Lem1} we find that  every vertex  of $A_{p^i}$ is adjacent to every  vertex of $A_{q^j}$.
	\\
	Also Lemma \ref{Adjacency} indicates that if  $1\le i\le \alpha$, $1\le j\le \beta$ with $i+j\neq \alpha+\beta$, then no vertex of $A_{p^i q^j}$ is adjacent to any other vertex of $G_2$.
	If we draw the graph $G_2$ with the vertex partitions as given in Equation \ref{partition}, it looks like the following (Figure \ref{Fig1})
	\begin{figure}[H]
		
		\centering
		\begin{tikzpicture}[Dotted/.style={
			line width=0.75pt,
			dash pattern=on 0.01\pgflinewidth off #1\pgflinewidth,line cap=round,
			shorten >=0.5em,shorten <=0.5em},
		Dotted/.default=5]
		\matrix[matrix of math nodes,nodes={circle,draw,minimum size=4.5em},
		column sep=4em,row sep=1ex](mat) {
			A_p & A_q \\
			A_{p^2} & A_{q^2} \\[2em]
			A_{p^\alpha} & A_{q^\beta} \\
		};
		\draw[Dotted] (mat-2-1) -- (mat-3-1);
		\draw[Dotted] (mat-2-2) -- (mat-3-2);
		\foreach \X in {1,2,3}
		{\foreach \Y in {1,2,3}
			{\draw (mat-\X-1) -- (mat-\Y-2);}}
		\matrix[matrix of math nodes,nodes={circle,draw,minimum size=2.5em},
		column sep=1em,row sep=1em,below=2em of mat,xshift=2em,
		column 2/.style={column sep=4.5em}](mat2) {
			A_{pq} & A_{p^2q} & A_{p^\alpha q}\\
			A_{pq^2} & A_{p^2q^2} & A_{p^\alpha q^2}\\[2em]
			A_{pq^\beta} & A_{p^2q^\beta} & A_{p^{\alpha-1} q^\beta}\\
		};
		\foreach \X in {1,2,3}
		{\draw[Dotted] (mat2-\X-2) -- (mat2-\X-3);
			\draw[Dotted] (mat2-2-\X) -- (mat2-3-\X);}
		\end{tikzpicture}

		\caption{$G_2$ for $n=p^\alpha q^\beta$}
		\label{Fig1}

	\end{figure}
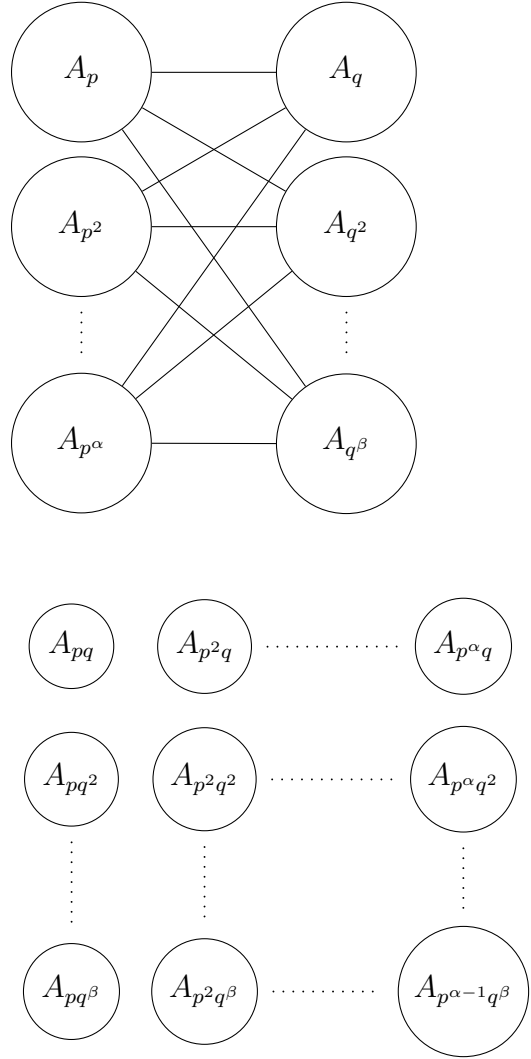
	(A solid line in the figure  indicates that  each vertex of $A_{d_i}$ is adjacent to each vertex of $A_{d_j}$.
	No line between two nodes $A_{d_i}$ and $A_{d_j}$ indicates that no vertex of $A_{d_i}$ is adjacent to any vertex of $A_{d_j}.$
	).
	\\
	\\
	Let $G_{21}$ be the \textit{induced subgraph} of $G_2$ on the set  $ A_{p}\cup A_{p^2}\cup \cdots \cup A_{p^\alpha} $ and  $G_{22}$ be the \textit{induced subgraph} of $G_2$ on the set $ A_{q}\cup A_{q^2}\cup \cdots \cup A_{q^\beta}$.
	\\
	\\
	Now the number of elements in $ A_{p}\cup A_{p^2}\cup \cdots \cup A_{p^\alpha} $ is $\sum_{i=1}^\alpha |A_{p^i}|.$
	Hence
	\begin{equation}\label{calculation}
	\begin{split}
	\sum_{i=1}^\alpha |A_{p^i}|=|A_p|+|A_{p^2}|+\cdots+|A_{p^\alpha}|
	=\phi(p^{\alpha-1}q^\beta)+\phi(p^{\alpha-2}q^\beta)+\cdots+\phi(q^\beta)
	\\
	=p^{\alpha-1}(1-\frac{1}{p})q^\beta(1-\frac{1}{q})+p^{\alpha-2}(1-\frac{1}{p})q^\beta(1-\frac{1}{q})+\cdots\\
	+\cdots+p(1-\frac{1}{p})q^\beta(1-\frac{1}{q})+q^\beta(1-\frac{1}{q})
	\\
	=q^\beta(1-\frac{1}{q})(1-\frac{1}{p})\{p+p^2+\cdots+p^{\alpha-1}\}+q^\beta(1-\frac{1}{q})
	\\
	=q^\beta(1-\frac{1}{q})(1-\frac{1}{p})\bigg(\frac{p(p^{\alpha-1}-1)}{p-1}\bigg)+q^\beta(1-\frac{1}{q})\\
	=q^\beta(1-\frac{1}{q})(p^{\alpha-1}-1+1)
	=q^{\beta-1}p^{\alpha-1}(q-1).
	\end{split}
	\end{equation}
	
	Again the number of elements in the set $ A_{q}\cup A_{q^2}\cup \cdots \cup A_{q^\beta}$ is $\sum_{i=1}^\beta |A_{q^i}|.$
	Using similar calculations as in Equation \ref{calculation}, we find that 
	\begin{align}\label{cal}
	\sum_{i=1}^\beta |A_{q^i}|=p^{\alpha-1}q^{\beta-1}(p-1).
	\end{align}

	The  vertices of $G_2$ which are not adjacent to any other vertex in $G_2$ are the members of the set
	\\
	$ \bigg( \cup_{j=1}^\beta A_{pq^j} \bigg )\cup \bigg( \cup_{j=1}^\beta A_{p^2q^j}\bigg )\cup \cdots \cup\bigg (\cup_{j=1}^{\beta-1} A_{p^\alpha q^j}\bigg )$. 
	Using Equations \ref{calculation} and \ref{cal}, the number of such vertices denoted by $t$ equals
	\begin{align*}
	\begin{split}
	t=p^\alpha q^\beta- \phi(p^\alpha q^\beta)-1-p^{\alpha-1} q^{\beta-1}(p-1)-p^{\alpha-1} q^{\beta-1}(q-1)
	\\
	=p^\alpha q^\beta-1-p^{\alpha-1} q^{\beta-1}\{(p-1)(q-1)+p-1+q-1\}
	\\
	=p^\alpha q^\beta-1-p^{\alpha-1} q^{\beta-1}(pq-1)
	=p^{\alpha-1} q^{\beta-1}-1.
	\end{split}
	\end{align*}
	Clearly the \textit{induced subgraph} of $G_2$ on $p^{\alpha-1} q^{\beta-1}-1$ vertices is a null graph.
	\\
	\\
	\\
	Since every vertex of the graph $G_{21}$ is adjacent to every vertex of the graph $G_{22}$ and the remaining vertices of $G_2$ are not adjacent to any other vertex, the following is evident
	\begin{equation}\label{Claim}
	G_2=(G_{21}\lor G_{22}) \cup \Theta_t.
	\end{equation}
	
	Using Equations \ref{calculation} and \ref{cal} and theorem \ref{Thjoin} we obtain
	\begin{equation}\label{T1}
	\begin{split}
	\mu((G_{21}\lor G_{22}),x)
	=x\bigg(x-p^{\alpha-1}q^{\beta-1}(q-1)\bigg)^{p^{\alpha-1}q^{\beta-1}(p-1)-1}\\
	\times\bigg(x-p^{\alpha-1}q^{\beta-1}(p-1)\bigg)^{p^{\alpha-1}q^{\beta-1}(q-1)-1}
	\bigg(x-(p^{\alpha-1}q^{\beta-1}(p+q-2))\bigg).
	\end{split}
	\end{equation}
	
	Using Theorem \ref{Thunion}, Equation \ref{Claim}  and Equation \ref{T1} we obtain,
	
	\begin{equation}\label{T2}
	\begin{split}
	\mu(G_2,x)=x^t\times\mu((G_{21}\lor G_{22}),x)
	=x^{t+1}\bigg(x-p^{\alpha-1}q^{\beta-1}(q-1)\bigg)^{p^{\alpha-1}q^{\beta-1}(p-1)-1} 
	\\
	\times\bigg(x-p^{\alpha-1}q^{\beta-1}(p-1)\bigg)^{p^{\alpha-1}q^{\beta-1}(q-1)-1}
	\bigg(x-(p^{\alpha-1}q^{\beta-1}(p+q-2))\bigg).
	\end{split}
	\end{equation}
	
	Using Equation \ref{T2} in Equation \ref{main} we have 
	
	\begin{equation*}
	\begin{split}
	\mu(\Gamma(\mathbb{Z}_n),x)=x(x-n)^{\phi(n)}\mu\bigg (G_2,x-\phi(n)\bigg)
	\\
	=x(x-n)^{\phi(n)}\bigg(x-\phi(n)\bigg)^{t+1}\bigg(x-p^{\alpha-1}q^{\beta-1}(q-1)-\phi(n)\bigg)^{p^{\alpha-1}q^{\beta-1}(p-1)-1}\\
	\times\bigg(x-p^{\alpha-1}q^{\beta-1}(p-1)-\phi(n)\bigg)^{p^{\alpha-1}q^{\beta-1}(q-1)-1}
	\bigg(x-(p^{\alpha-1}q^{\beta-1}(p+q-2)-\phi(n))\bigg).
	\end{split}
	\end{equation*}
	
	Thus the eigenvalues of $\Gamma(\mathbb Z_n)$ are $n$ with multiplicity $\phi(n)$, $(t+1)(p-1)+\phi(n)$ with multiplicity $(t+1)(q-1)-1$, $(t+1)(q-1)+\phi(n)$ with multiplicity $(t+1)(p-1)-1$, $\phi(n)$ with multiplicity $t+1$ and $(t+1)(p+q-2)+\phi(n),0$ each with multiplicity $1$.
	
\end{proof}

Using Corollary \ref{prime} and Theorems  \ref{pm} and \ref{prod}  the following is evident,

\begin{theorem}\label{integral}
	If   $n=p^\alpha q^\beta$ where $p,q$ are primes and $ \alpha, \beta$ are non-negative integers, then $\Gamma(\mathbb Z_n)$ is Laplacian Integral.
\end{theorem}

\section{Algebraic Connectivity and Vertex Connectivity of $\Gamma(\mathbb{Z}_n)$}
\label{Sec4}

In this section we investigate the algebraic connectivity($\lambda_{n-1}$) and vertex connectivity($\kappa$) of $\Gamma(\mathbb Z_{n})$  for any $n>2$.
We also show that $\lambda_{n-1}$ and $\kappa$ are equal for any $n>2$.

\begin{lemma}\label{AC}
	If $n> 2$, then $\phi(n)$ is an eigenvalue of $\Gamma(\mathbb Z_n)$ with multiplicity atleast $1$.
\end{lemma}

\begin{proof}
	
Since $0$ is always an eigenvalue of the Laplacian matrix of a given graph $G$, so the Laplacian matrix of the graph $G_2$ also has $0$ as an eigenvalue.
Using Equation \ref{main}, $x-\phi(n)$ is a factor of $\mu(G_2,x-\phi(n))$ which in turn implies
\begin{align}
\begin{split}\label{g}
\mu(\Gamma(\mathbb{Z}_n),x)=x(x-n)^{\phi(n)}\mu(G_2,x-\phi(n))
 \\
=x(x-n)^{\phi(n)}(x-\phi(n)) g(x-\phi(n)) \text{ where }
\end{split}
\end{align}
$g(x)$ is a polynomial of degree $n-\phi(n)-2.$
Hence $\phi(n)$ is an eigenvalue of $\Gamma(\mathbb Z_n)$ with multiplicity atleast $1$.	
\end{proof}

\begin{theorem}\label{ACjoin}
	$\lambda_{n-1}(\Gamma(\mathbb Z_n))=\phi(n).$
\end{theorem}

\begin{proof}
	Using Lemma \ref{AC}, $\phi(n)$ is an eigenvalue of $\Gamma(\mathbb Z_n)$.
	Since the smallest root of the  polynomial $g(x-\phi(n))$ in Equation \ref{g} is $\phi(n)$  and $0<\phi(n)<n$, we conclude that the second smallest root of $\mu(\Gamma(\mathbb{Z}_n),x)$ is $\phi(n)$ which implies that $\lambda_{n-1}(\Gamma(\mathbb Z_n))=\phi(n)$.
	
\end{proof}

\begin{theorem}\label{VC=AC}
	For all $n> 2$, $\kappa(\Gamma(\mathbb Z_n))=\lambda_{n-1}(\Gamma(\mathbb Z_n))=\phi(n)$.
\end{theorem}

\begin{proof}
	Using Equation \ref{join} we find that $\Gamma(\mathbb Z_n)=(G_2\cup \Theta_1)\lor K_{\phi(n)}$.
	If we take $\mathcal{G}_1= G_2\cup \Theta_1$ and $\mathcal{G}_2=K_{\phi(n)}$, we find that $\mathcal{G}_1$ is a disconnected graph on $n-\phi(n)$ vertices and $\mathcal{G}_2$ is a graph on $\phi(n)$ vertices.
	Clearly $\lambda_{n-1}(\mathcal{G}_2)=\lambda_{n-1}(K_{\phi(n)})=\phi(n)$.
	We find that if we assume  $\kappa(\Gamma(\mathbb Z_n))=\phi(n)$, then all the conditions of Theorem \ref{VCequalsAC} along with the  inequality $\lambda_{n-1}(\mathcal{G}_2)\geq 2\kappa(G)-n$ is satisfied.
	Hence we conclude that $\kappa(\Gamma(\mathbb Z_n))=\lambda_{n-1}(\Gamma(\mathbb Z_n))=\phi(n)$.
\end{proof}

\section{Largest \& Second Largest Eigenvalue of $\Gamma(\mathbb Z_n)$}
\label{Sec5}

In  this section we discuss about the second largest eigenvalue $\lambda_{2}$ of $\Gamma(\mathbb Z_n)$ which in turn helps us to find certain information about the largest eigenvalue $\lambda_1$ of  $\Gamma(\mathbb Z_n)$.
\\
We first study the connectivity of $G_2$.

\begin{theorem}\label{connected}
	The graph $G_2$ is connected if and only if $n$ is a product of distinct primes.
\end{theorem}

\begin{proof}
	
Let   $n=p_1^{\alpha_1}p_2^{\alpha_2}\cdots p_m^{\alpha_m}$ where $p_i$ are distinct primes and
 $\alpha_i$ are positive integers, $1\leq i\leq m$.
\\
We first assume that $G_2$ is connected.
In order to show that $n$ is a product of distinct primes, we  prove that $\alpha_i=1$ for all $1\le i\le m$.
Assume the contrary that $\alpha_i>1$ for atleast one $i$. 
Without loss of generality we take  $\alpha_1>1$.
We consider the vertex $a=p_1p_2p_3\cdots p_m$ of $G_2$.
Clearly $a\neq 0$ as $\alpha_1>1$.
Consider any other vertex of $G_2$ say $w$.
Since $V(G_2)=\cup_{i=1}^w A_{d_i}$ where $A_{d_i}$ has been defined in Equation \ref{divisor},  $w\in A_{d_i}$ for some positive proper divisor $d_i$ of $n$.
Thus $\gcd(w,n)=d_i$.
Also $a\in A_{p_1p_2p_3\cdots p_m}$.
Since $\gcd(d_i,p_1p_2p_3\cdots p_m)\neq 1$, using Lemma \ref{Adjacency} we conclude that $w$ is not adjacent to $a$.
Since $w$ is arbitrary, we find that the vertex $a\in G_2$ is not adjacent to any other vertex of $G_2$ which contradicts the fact that $G_2$ is connected.
Hence our assumption that $\alpha_1>1$ is false.
Thus $\alpha_i=1$ for all $1\le i\le m$ which proves that $n$ is a product of distinct primes.
\\
\\
Conversely we assume that $n$ is a product of distinct primes.
In order to show that $G_2$ is connected we choose two arbitrary distinct vertices $x_i,x_j\in G_2$.
Then $x_i\in A_{d_i}$ and $x_j\in A_{d_j}$ for some proper positive divisor $d_i,d_j$ of $n$.
We consider the following two cases which may arise.
\begin{enumerate}
\item [1.] $\gcd(d_i,d_j)=1$
		Using Lemma \ref{Adjacency}, $x_i$ and $x_j$ are adjacent in $G_2$.
		\item [2.]  $\gcd(d_i,d_j)\neq 1$
		
	Since $\gcd(d_i,d_j)\neq 1$, $d_i,d_j$ have a prime factor in common.
	Since $n$ is a product of distinct primes, so there exists a prime factor $p_1$ of $n$ such that $\gcd(d_i,p_1)=1$.
	Also it is possible to choose another prime factor $p_2\neq p_1$ of $n$ such that $\gcd(d_j,p_2)=1$.
	Since $\gcd(p_1,p_2)=1$, if we choose $x_{p_1}\in A_{p_1}$ and $x_{p_2}\in A_{p_2}$ then $x_{p_1}$ is adjacent to $x_{p_2}$.
	Thus using Lemma \ref{Adjacency} we obtain a path of length  $3$ from $x_i$ to $x_j$ given by 
	\begin{align*}
	x_i\rightarrow x_{p_1}\rightarrow x_{p_2}\rightarrow x_j
	\end{align*}
	\end{enumerate}	
Combining cases $1$ and $2$ we find that any two vertices of $G_2$ are either adjacent or there exists a path between them which implies that $G_2$ is connected when $n$ is a product of distinct primes.	
\\
Thus $G_2$ is connected if and only if $n$ is a product of distinct primes.

\end{proof}

Now we investigate the connectivity of the complement of the graph $G_2$ (denoted by $G_2^c$) when $n$ is a product of distinct primes.
\\
\\
When $n$ is a product of two distinct primes, i.e. $n=pq$, then $n$ has only two distinct proper positive divisors namely $p$ and $q$.
Thus $V(G_2)=A_p\cup A_q$.
Since $\gcd(p,q)=1$, using Lemma \ref{Adjacency}, $G_2$ becomes
\begin{figure}[H]
	\centering
	\begin{tikzpicture}[circ/.style={draw, shape=circle,minimum size=3em,
		path picture={\fill (20:0.5em) circle[radius=0.75pt]
			(140:0.5em) circle[radius=0.75pt]  (260:0.5em) circle[radius=0.75pt];
	}}]
	\node[circ,label=below:$A_p$] (v0) at (0,0) {};
	\node[circ,label=below:$A_q$] (v1) at (0:4) {};
	\draw (v0) -- (v1) ;
	\end{tikzpicture}

	\caption{ $G_2$ when $n=pq$}
	\label{Pic}
\end{figure}
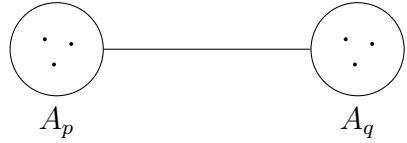

(Here the solid line indicates that  each vertex of $A_p$ is adjacent to each vertex of $A_q$).
\\
Clearly $G_2^c$ is disconnected when $n=pq$.
\\
In the next theorem, we investigate the connectivity of $G_2^c$ when $n$ is a product of more than $2$ distinct primes.

 \begin{theorem}\label{complement}
	If $n$ is a product of more than two distinct primes, then  $G_2^c$ is connected.
\end{theorem}

\begin{proof}

Let $n=p_1p_2p_3\cdots p_m$ where $p_i$ are distinct primes and $m>2$.
Let $x_i,x_j$ be two distinct vertices  of $G_2^c$.
Then $x_i\in A_{d_i}$ and $x_j\in A_{d_j}$ where $d_i,d_j$ are positive proper divisors of $n$.
 We consider the following two cases:
\begin{enumerate}
	\item [1.] $\gcd(d_i,d_j)\neq 1$

Using Lemma \ref{Adjacency}, $x_i\in A_{d_i}$ is not adjacent to $x_j\in A_{d_j}$ in $G_2$	which implies that $x_i\in A_{d_i}$ is adjacent to $x_j\in A_{d_j}$ in $G_2^c$.
	
	\item [2.] $\gcd(d_i,d_j)= 1$
	
	Using Lemma \ref{Adjacency}, $x_i\in A_{d_i}$ is not adjacent to $x_j\in A_{d_j}$ in $G_2$. 
	Let $p_1$ be a prime factor of $d_i$ and $p_2$ be a prime factor of $d_j$.	
	Since $n$ is a product of more than two distinct primes, so $p_1p_2$ is  a  positive proper divisor of $n$.
	Hence using Lemma \ref{Adjacency}, there exists $y\in A_{p_1p_2}$ such that $x_i,x_j$ are not adjacent to $y$ in $G_2$.	Thus $y$ is adjacent to both $x_i$ and $x_j$ in $G_2^c$ and hence there exists a path of length $2$ given by $x_i\rightarrow y\rightarrow x_j$ from $x_i$ to $x_j$ in $G_2^c$.

\end{enumerate}
Combining cases $1$ and $2$ we find that any two vertices of $G_2^c$ are either adjacent or there exists a path between them which implies that $G_2^c$ is connected when $n$ is a product of more than two distinct primes.	

\end{proof}
	
\begin{theorem}\label{Th7}

	$\lambda_2(\Gamma(\mathbb Z_n))\leq n-1$ where equality holds if and only if  $n$ is a product of two distinct primes.

\end{theorem}
	
\begin{proof}
	Let $\lambda_{1}(G_2)$ denote the largest eigenvalue of the Laplacian matrix of $G_2$.	
	Using Equation \ref{main} of theorem \ref{Th1} it is evident that the second largest eigenvalue of $\Gamma(\mathbb Z_n)$  is the largest eigenvalue of the Laplacian matrix of $G_2$ which implies 
	\begin{align*}
     \lambda_2(\Gamma(\mathbb Z_n))=\lambda_{1}(G_2)+\phi(n).
	\end{align*}
	Since $G_2$ is a graph on $n-\phi(n)-1$ vertices, using Theorem \ref{Th3} we have
	$\lambda_{1}(G_2)\le n-\phi(n)-1$ where equality holds if and only if $G$ is connected and $G_2^c$ is disconnected.
	\\
	\\
	Using Theorems \ref{connected} and \ref{complement} we find that $G_2$ is connected if and only if $n$ is a product of distinct primes and $G_2^c$ is disconnected if $n$ is a product of two primes.
	Thus
	\begin{align*}
	\lambda_2(\Gamma(\mathbb Z_n))=\lambda_{1}(G_2)+\phi(n)\leq (n-\phi(n)-1)+\phi(n)=n-1
	\end{align*}
	
	where equality holds if and only if $n$ is a product of two primes.
	
\end{proof}

	\begin{theorem}\label{mult spectral}
 For any $n> 2$,  $\lambda_1(\Gamma(\mathbb Z_n))=n$ has multiplicity exactly $\phi(n)$.
	\end{theorem}
	\begin{proof}
		Using Theorem \ref{Th7},  $\lambda_2(\Gamma(\mathbb Z_n))\le n-1$.  Thus from Equation \ref{main}
		of Theorem \ref{Th1}, we conclude that  $\lambda_1=n$ has multiplicity exactly $\phi(n)$.
	\end{proof}
	
\begin{theorem}\label{multAC}
	If $n=\prod_{i=1}^{m}p_i^{\alpha_i}$ where $p_i$ are distinct  primes and $\alpha_i$ are positive integers, then  $\phi(n)$ is an eigenvalue of $\Gamma(\mathbb{Z}_n)$ with multiplicity $\dfrac{n}{\prod_{i=1}^m p_i}$.
\end{theorem}

\begin{proof}
	Let us first assume that  $n$ is a product of distinct primes i.e. $n=p_1p_2\cdots p_m$.
	Using  Lemma \ref{connected}, $G_2$ is connected and hence $0$ is an eigenvalue of $L(G_2)$ with multiplicity $1$ which in turn using Equation \ref{main} implies that $\phi(n)$ is an eigenvalue of $\Gamma(\mathbb{Z}_n)$ with multiplicity $1$.  Since $\dfrac{n}{\prod_{i=1}^m p_i}=1$, the theorem holds true.
	\\
	We now assume that  $n$ is not a product of distinct primes, i.e. $\alpha_i>1$ for atleast one $1\le i\le m$.
	The set of vertices of $G_2$ in $\langle p_1p_2\cdots p_m \rangle\setminus \{0\}$ are not adjacent to any other vertex in $G_2$.
	Since the set $\langle p_1p_2\cdots p_m \rangle\setminus \{0\}$ has $\dfrac{n}{\prod_{i=1}^m p_i}-1$ elements, the graph $G_2$ has $\dfrac{n}{\prod_{i=1}^m p_i}$ connected components.
	Hence $0$ is an eigenvalue of $L(G_2)$ with multiplicity $\dfrac{n}{\prod_{i=1}^m p_i}$ which in turn using Equation \ref{main} implies that $\phi(n)$ is an eigenvalue of $\Gamma(\mathbb{Z}_n)$ with multiplicity $\dfrac{n}{\prod_{i=1}^m p_i}$.
\end{proof}

\subsection{Vertex connectivity of $G_2$}\label{Sec50}
	
This section deals with  the vertex connectivity of $G_2$.
Since $G_2$ is connected if and only if $n$ is a product of distinct primes, we discuss $\kappa(G_2)$ when $n=p_1p_2\cdots p_m$ where $p_i$, $1\le i\le m$ are distinct primes.
\\
\\
We first give an  example to illustrate $\kappa(G_2)$.
\begin{example}
	Suppose $n=3\times 5\times 7$.
	Consider the vertex $15$ in $G_2$.
	Consider the set  $\langle 7\rangle \setminus\{0\}=\{7k: 1\leq k\leq 14\}.$
	We notice that $15$ is adjacent only to the following vertices $\{7,14,28,49,56,77,91,98\}$. 
	Thus the set  $\{7,14,28,49,56,77,91,98\}$ is a separating set of $G_2$.
	\\
	The elements of the set $\{7,14,28,49,56,77,91,98\}$ are of the form $\{7k: 1\leq k\leq 14, \gcd(k,14)=1\}.$
We find that $\kappa(G_2)\le 8=\phi(15).$
\end{example}

We prove the above formally in the following theorem:

\begin{theorem}
	If $n=p_1p_2\dots p_m$, then $\kappa(G_2)\leq \phi(p_1p_2p_3\dots p_{m-1})$.
\end{theorem}
\begin{proof}
	
	We first verify the result when $n$ is a product of two distinct primes.
	The graph of $G_2$ when $n$ is a product of two distinct primes has been shown in Figure \ref{Pic}.
	If $n=p_1p_2$, then $G_2$ is the join of two disconnected graphs having vertex sets as $A_{p_1}$ and $A_{p_2}$  and hence $\kappa(G_2)=\min \{|V_1|,|V_2|\}=\min\{p_1-1,p_2-1\}=p_1-1=\phi(p_1)$ and hence our result holds.
	\\
	\\	
	When $n=\prod_{i=1}^m p_i$ where $m>2$, then the vertex $p_1p_2\cdots p_{m-1}$ of the graph $G_2$ is adjacent only to those members $\textbf{a}$ of the set $\langle p_m\rangle \setminus \{0\}$ such that $\gcd(\textbf{a},p_1p_2\cdots p_{m-1})=1$.
	The number of those elements $\textbf{a}$ such that  $\gcd(\textbf{a},p_1p_2\cdots p_{k-1})=1$ equals $\phi(p_1p_2\cdots p_{m-1})$.
	Since the vertices $\textbf{a}$  for which $\gcd(\textbf{a},p_1p_2\cdots p_{m-1})=1$ becomes a \textit{separating set} of the graph $G_2$, the result follows.
\end{proof}	
	
\section{ Problems}
\label{Sec6}
In this section we pose some problems for further research.
\\
Using Theorem \ref{integral}, we observe that  $\Gamma(\mathbb Z_n)$ is Laplacian Integral for  $p^\alpha q^\beta$ where $p,q$ are  primes and $\alpha,\beta$ are non-negative integers.
Since it is quite motivating to find those graphs which are \textit{Laplacian Integral}, we pose the following problem:

\begin{problem}
	Is it true that $\Gamma(\mathbb Z_n)$ is \textit{Laplacian Integral} if and only if $n=p^\alpha q^\beta$ where $p,q$ are  primes and $\alpha,\beta$ are non-negative integers? If not, then find all $n$ such that  $\Gamma(\mathbb Z_n)$ is \textit{Laplacian Integral}.
	
\end{problem}
\noindent
Again, in subsection \ref{Sec50} we have provided an upper bound on the vertex connectivity of the graph $G_2$ which is an induced subgraph of $\Gamma(\mathbb Z_n)$.
Though we have provided an upper bound on $\kappa(G_2)$,  the readers are  encouraged to calculate the exact value of  $\kappa(G_2)$ if possible.
Thus we ask the following:

\begin{problem}
	If $n=p_1p_2\cdots p_m$ where $p_1<p_2<\cdots< p_m$ are primes, find $\kappa(G_2)$.
\end{problem}


\end{document}